\documentclass[11pt]{amsart}
\usepackage[english]{babel}
\usepackage{graphicx}
\usepackage{datetime}
\usepackage{amssymb,amsmath}
\usepackage{caption}
\usepackage{subcaption}
\usepackage{hyperref}

\newcommand{\R}{\ensuremath{\mathbb R}}
\newcommand{\Rn}{\R^n}
\newtheorem{theorem}{Theorem}

\newtheorem{lemma}[theorem]{Lemma}
\newtheorem{proposition}[theorem]{Proposition}

\newtheorem{summary}[theorem]{Summary}

\newtheorem{defn}[theorem]{Definition}
\newtheorem{rem}[theorem]{Remark}

\def \sP {{\mathcal{P}}}

\def \cont {{\mathcal{C}}}
\def \grad {\nabla}
\def \diag {\text{diag}}
\def \span {\text{span}}

\theoremstyle{definition}

\begin{document}

\title{Double Descent and Intermittent Color Diffusion \\ for Global Optimization and landscape
exploration}

\author{Luca Dieci, Manuela Manetta, Haomin Zhou}
\thanks{
The second and third authors work was supported under
NSF Awards DMS-1419027, DMS-1620345 and ONR Award N000141310408.}



\begin{abstract}
In this work, we present a method to explore the landscape of a smooth
potential in the search of global minimizers,
combining a double-descent technique and a basin-escaping technique
based on intermittent colored diffusion.  Numerical results illustrate
the performance of the method.
\end{abstract}

\maketitle

\pagestyle{myheadings}
\thispagestyle{plain}
\markboth{LUCA DIECI AND MANUELA MANETTA AND HAOMIN ZHOU}{Double Descent and Intermittent Color Diffusion}

\section{Introduction}\label{Intro}
Let $g:\Rn\to \R$, $n\ge 1$, be a sufficiently smooth function (say, $\cont^\infty$);
we call $g$ the ``potential,'' or ``objective function.''
Let $\grad g$ be the gradient of $g$, and $H$ be the Hessian.  
Finally, we also let $G: \Rn \to \R^+$ be defined as
$G=\frac12 (\grad g)^T (\grad g)$; we call $G$ the ``auxiliary potential.''
Our goal is to minimize $g$.

Finding global minimizers for a general objective function $g$ is one of the
oldest and most
challenging problems in applied mathematics.   Whereas it is at times possible
to exploit a-priori knowledge for specific potentials, it remains
an outstanding task to devise effective general optimization
strategies which can be applied to a general problem.
In the literature, one finds extensive collections of
real-world potentials, arising from chemistry,
physics, mathematics, etc., as well as artificial problems.   Many challenging
problems in the
first group are obtained as models of interatomic forces, and are
distinguished by having a reduced region of interest, expensive computation of
the potential and its gradient, and full (not sparse) Hessians.  Problems in
the second group are useful to
validate an optimization technique, to illustrate it, and to have objective functions with
selectively distinguished features: a single
global minimum, multiple global minima in the presence of many local minima (in which case
any deterministic technique will be trapped in the basin of attraction of a minimizer
without being able to escape it), long narrow valleys (which will slow down the search
process), and flat surfaces.  Unsurprisingly, some methods perform well on some problems,
and poorly on others, and -aside from knowing ahead of time what method to use on
a specific potential- one is left wandering on what to use for a given problem.  

We are often confronted with
this frustrating state of affairs when teaching a course on 
numerical methods for optimization.  Even absolutely marvelous textbooks (e.g.,
\cite{DS}) are ultimately having to accept some 
uncertainties, and to deal with fine-tuning of parameters, and empirical choices.  To be fair,
these difficulties are intrinsic to the task at hand, and surely not
the result of negligence.  So, when we teach such a course,
we end up teaching local techniques, maybe continuation and embedding techniques,
emphasize gradient descent and Newton's method and their variants, 
stress convex or maybe polynomial functions, but in the end we
fail at providing rigorous answers to the recurring questions of alert students
relative to a general smooth function $g$:
``how do we know that we found the global minimum?  how do we know that we have
visited the interesting regions of configuration space?''   We don't know, and most
likely we will not know for the foreseeable future.  Indeed, borrowing a painstaking
and extensive search of the configuration space, we have few hints to offer
to our students for answering their questions above.
Motivated by our classroom experience, 
one of our purposes in this work is to present methods and ideas that can be
taught in a numerical optimization course.  That said, quite honestly,
we have no pretense that our work is an answer to
the above questions,  but we are hopeful to be taking a (small) step in the
right direction.


Let us immediately stress that we are putting forth some ideas for a general purpose
method, one which does not rely on the specific properties of the potential.
With this purpose in mind, we may recall that the main components of a global minimization
algorithm are: to explore the landscape, and to locate minima.  
The shape of the level sets of $g$ is of course dictating the nature
of the landscape: flat, rough, predictable, crowded with minima. 
At the same time, the shape of the level sets of $G$ is more directly
responsible for properly identifying the basins of attraction of the minima of $g$
for important techniques, such as Newton's method.  
Naturally, the level sets of $g$ and of $G$ are often of a very different nature;
for example, in   \ref{MoleiConts}
we show them relative to the function of Example \ref{Ex:Molei}.

\begin{figure}[hbt]
\begin{subfigure}{0.45\textwidth}
\includegraphics[scale=0.4]{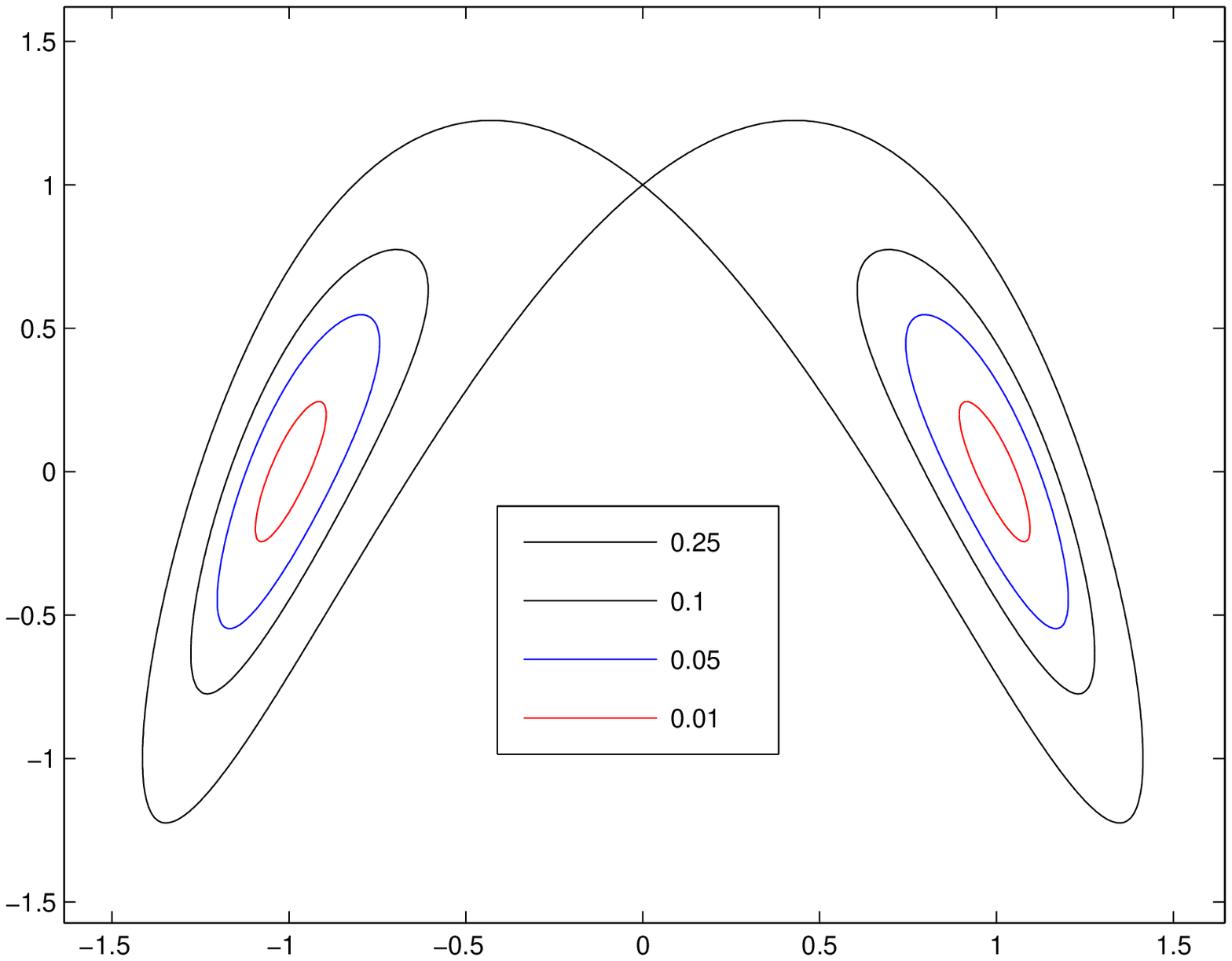}
\end{subfigure}
\begin{subfigure}{0.45\textwidth}
\includegraphics[scale=0.4]{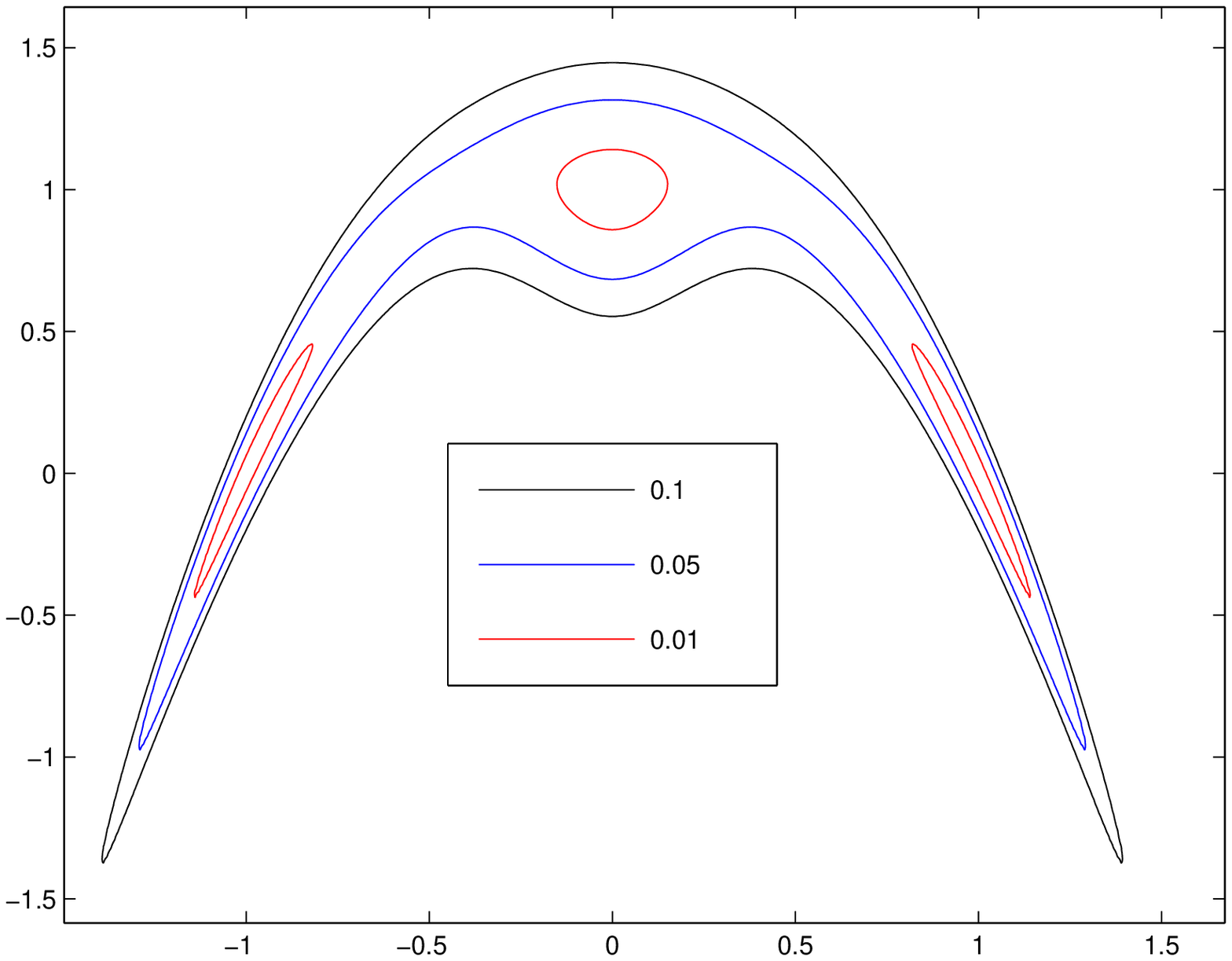}
\end{subfigure}
\caption{\small{Level sets for $g$ (on the left) and $G$ (on the right);
Example \ref{Ex:Molei}}}\label{MoleiConts}
\end{figure}

Regardless, unfortunately,
graphical insight provided by the level sets is all but unavailable for problems
in several variables, and exploration of the landscape remains a mix of
randomization and subdivision ideas.  Indeed, the most widely
adopted exploratory technique appears to be Monte-Carlo, which is then
combined with localization methods such as gradient descent,
Newton's method, and their modifications. For example, the popular ``basin-hopping''
methods randomly perturb a minimum and search for other minima by gradient descent 
(e.g., see \cite{LennJ}).  And, of course, the best known instance of a method
exploiting stochastic effects is simulated annealing, which makes
use of a probability function to decide how to move in the search space; briefly, the 
method can be described as follows: call $x$ the current state, $x_{new}$
a randomly selected neighbor of $x$, and let the probability function be given by
\begin{equation*}\label{SA}
P(x,x_{new},T)=\left\{ \begin{array}{ll}
1, & \text{ if }  g(x_{new})<g(x),\\
exp\left(\frac{-[g(x_{new})-g(x)]}{T}\right), & \text{ if } g(x_{new})>g(x),
\end{array}\right.
\end{equation*}
where $T$ is called the temperature and is a function of the ratio between the current iteration
number and the total number of iterations allowed.  Given this setting, a random number $r$ is generated:
if the probability to move from $x$ to $x_{new}$ is greater than $r$, then 
the new state is accepted,  otherwise it is rejected; note that a downhill direction will
always be accepted, though one may also take uphill steps.
This method is inherently better suited
for discrete problems, and it is sequential in nature
(see \cite{BT} for more details on simulated annealing).

To a large extent, everyone needs to deal with the key two aspects above:
how to explore the landscape, and how to locate minima once the decision is made
that there is one in the vicinity of the iterates.  Our goal is to do so
while avoiding a purely Monte Carlo technique, but rather using
the information gathered at the critical points to move in an
educated way in the landscape exploration.

In this work we introduce a method which uses a combination of new techniques, namely a
double-descent method to search for minima, and an intermittent colored diffusion
technique to escape the basins of attraction of critical points. 

All along, we will tacitly assume that critical points (i.e., values $x$ where $G(x)=0$)
are simple, in the sense that the Hessian is invertible there.  
In particular, at minima, the associated Hessian will be positive definite.
Nevertheless, the proposed method will also be
able to solve problems where at the critical point the Hessian has eigenvalues equal to $0$.

We conclude this introduction with some practical considerations.

\begin{itemize}
\item[(i)]  Although we are considering  an unconstrained minimization problem, the
case of constrained optimization is of course also important, and we expect to
consider it in the future.
\item[(ii)]
Of the many minimization methods proposed during the years, some
use only gradient information, some also Hessian, and some use only
functional evaluations (so-called direct search algorithms).  Whereas,
of course, the specific problem at hand may inhibit using the gradient and/or
the Hessian, we will assume that these are available to us.
In fact, in our technique, we make use of repeated eigen-decomposition of the Hessian.
Of course, this is an expense which may be prohibitive for truly large problems,
though by today standards it is easily doable for dimensions of up
to a few hundreds.  It is not by coincindence
that a lot of people have been concerned with efficient
updating of Hessian factorizations (e.g., the BFGS (Broyden-Fletcher-Goldfarb-Shanno) or 
the DFP (Davidon-Fletcher-Powell) updates); see \cite{GMW}.
\item[(iii)]
The prevailing wisdom (e.g., see the well known Levenberg-Marquardt
algorithm, trust-region methods, and the discussion in \cite{DS} and \cite{Kelley})
is to do Newton's
method near a minimizer.  Our technique is designed to automatically do Newton's
method as well, as we reach the neighborhood of a minimize, or of another critical
point.
\item[(iv)]
Many recent advances in global optimization (e.g., genetic algorithms,
direct search techniques, multiple random initializations) have found their place
in public domain software; e.g., see \cite{Holm} and the 
{\tt Matlab} Global Optimization Toolbox.  In particular, 
the latter contains two routines which we have used for cross-comparison of
our results: {\tt GlobalSearch} 
and {\tt simulannealbnd}.   The function {\tt simulannealbnd}
is the {\tt Matlab} implementation of simulated annealing.  
Instead, {\tt GlobalSearch} finds minimizers at different stages: 
first a local search (carried out by the function {\tt fmincon}) starts from an initial point $x_0$
provided by the user, and then a list of trial point is generated as potential starting points,
taking into account penalty functions, spherical basins of attractions, and run-time, to eventually
perform the local search from a large number of initial points.  (For fair comparison with
the results of our method, we used this function by
providing gradient and Hessian information). 
\item[(v)]
Finally, we must stress that it is very delicate to implement any method, and that methods
that look good on paper will not deliver according to expectations, 
if not properly implemented.  For this reason, we will detail our implementation
choices so that our results may be replicated.
\end{itemize}

A plan of this paper is as follows.  In section \ref{prelims} we give some background
material.  In section \ref{descents} we introduce the double-descent technique, and
in section \ref{DD-CID} we give the combined ``double-descent-color-intermittent-diffusion''
method (DD-CID, for short).  An overview of our numerical method is in
section \ref{glance}, and several numerical experiments are reported
in section \ref{Expers}.

\section{Preliminaries}\label{prelims}

\subsection{Intermittent Diffusion}\label{IntDiff}

In the recent work \cite{ID13}, the authors proposed a general methodology,
called intermittent diffusion (ID, for short), motivated by the
fact that the most widely used stochastic technique available for global optimization,
the simulated annealing mentioned before, needs a deterministic part to speed up the
convergence towards the minimizers. In order to do so, ID alternates between gradient descent and
diffusion processes, by turning on and off a white noise term.  In mathematical terms,
the ID methodology can be summarized by the following stochastic differential equation:
\begin{equation}\label{IDsde}
dx(t,w) = -\nabla g( x(t,w)) dt+ \sigma(t) dW(t), \; t\in  [0,+\infty]
\end{equation}
where $W(t)$ is Brownian motion in $\R^n$, $w$ is a random path in the Wiener space and $\sigma(t)$ is a 
piecewise constant function of time alternating between positive and zero values. 
In particular, when the noise is off ($\sigma(t)=0$), the method reduces to the gradient descent technique, 
leading the trajectory towards a minimizer of the potential,  when the noise is on ($\sigma(t) >0$), 
the trajectory should leave a neighborhood of the minimizer 
and, eventually, reach the basins of attraction of different minima.\\

The discontinous diffusion term is given in \cite{ID13} as
\begin{equation}\label{IDnoise}
\sigma(t)= \sum_{i=1}^N \sigma_i I_{[S_j,T_j]}(t)
\end{equation}
where $0=S_1<T_1<S_2<T_2 \dots <S_N<T_N<S_{N+1}=T$ and $I_{[S_j,T_j]}$ is the characteristic function 
of the interval $[S_j,T_j]$.  In \cite{ID13},
the length of the intervals $T_j-S_j$, and the constants $\sigma_j$ are supposedly chosen
randomly for all $j=1, \dots, N$;
therefore, when the characteristic function is $1$, the minimizer is perturbed by a positive random
number for a certain amount of time, 
and when the noise is off, namely $I_{[S_j,T_j]}=0$, the method reduces to gradient descent 
and slowly converges to a local minimum. 

As originally proposed, the ID approach 
is a general methodology, 
but to make it become a practical method requires making a lot of choices; for example, to
decide how long the diffusion process needs to be carried out.  
In our experimentation of this technique, we used the discrete analog of \eqref{IDsde}:
\begin{equation}\label{IDsdeDiscrete}
x_{k+1}= x_k-h \nabla g(x_k)+ +\sqrt{h} \sigma(t_k) W\ ,
\end{equation}
where $W\in N(0,1)^n$.  However, when using this technique, we
faced the need to adjust too many parameters based on the potential we were trying to minimize, 
and realized that there were some key aspects to be addressed:
\begin{itemize}
\item the local convergence towards minima, using the gradient descent, was too slow,
and a faster method (eventually, Newton-like) was desirable;
\item white noise based diffusion did not account for the local landscape of the potential,
and we eventually wanted to modify this with color noise diffusion;
\item criteria were needed to replace choosing the interval length randomly, finding instead a deterministic
criterion to switch the noise on and off.
\end{itemize}
We addressed all of the above concerns in the present paper.

In order to achieve our goal to build up a method which automatically adjusts to the 
optimization problem, we resorted to exploiting the Hessian's spectral information.

\subsection{$\Rn$ via the Hessian}
Below, we clarify the structure of regions of $\Rn$ where the eigenvalues
of $H$ have a specified signature (inertia).

\begin{defn}\label{SignHess}
Given a symmetric matrix $H$, the inertia of $H$ is the triplet 
$$\nu(H)=\{n_+(H),n_0(H),n_-(H)\},$$ where $n_+, n_0, n_-$, are the number
of positive, zero, or negative, eigenvalues of $H$, counted with their
multiplicities.   $H$ is called hyperbolic if $n_0(H)=0$.
\end{defn}
Observe that $H:\Rn \to \R^{n\times n}$ is a
smooth symmetric function of $n$ parameters, hence the reasonings below
are valid. 
\\
We will always order the eigenvalues of $H$ as
$\lambda_1\ge \lambda_2\ge\cdots \ge \lambda_n$, and $v_1,\dots ,v_n$, will
be the corresponding orthogonal eigenvectors.  According to $\nu(H)$, we
will also use the notation $V=[v_1,\  \dots\ , \ v_n]=[V_+ ,\  V_0 , \ V_-]$,
and will call $V_+$ the basis for the positively dominant subspace, or simply
(with improper language) the dominant subspace, etc..

\begin{lemma}\label{PropertiesHess}
Consider the set $\sP:=\{x\in \Rn\ : \ y^TH(x)y>0\ ,\, \forall y\in \Rn\}$.
We have the following properties:
\begin{enumerate}
\item $\sP$ is open.
\item $\sP=\cup_i \sP_i$, where each $\sP_i$ is open and connected and 
$\sP_i\cap \sP_j=\emptyset$ for $i\ne j$.
\item Each $\sP_i$ is path connected.
\end{enumerate}
\end{lemma}
\begin{proof}
(1) follows from these considerations.  
\begin{itemize}
\item[(i)]  The eigenvalues of the function $H$ are continuous functions of $x$.
(This is a standard result).
\item[(ii)]  If $A$ is a symmetric positive definite function,
and $B=B^T$ is such that $\|A-B\|_2<\lambda_n(A)$,
then $B$ is positive definite.  (This is also well known). 
\end{itemize}
Thus, if we take a value $x_0$ where $H(x_0)$ is pos-def, 
consider the smallest eigenvalue of $H$ as a function of $x$, and
use the continuity of $\lambda_n(\cdot)$, then
$H(\cdot)$ must remain positive definite for $x\in B_r(x_0)$ 
(an open ball centered at $x_0$, of radius $r$): $B_r(x_0)=\{y\in \Rn\ :\, \|x-y\|<r\}$,
or also $B_r(x_0)=\{x\in \Rn\ :\, x=x_0+\rho w,\ \|w\|=1,\ \rho<r\}$. \\
As far as (2), 
the observation is that the function $H$ ceases to be positive definite when 
$\lambda_n(x)=0$.  So, we define a set $\sP_i$ as that component of $\sP$ such
that for any two $x,y\in \sP_i$ there is a curve joining $x$ and $y$ such that
along this curve $\lambda_n>0$.  As above, $\sP_i$ is open, and thus
the set $\sP$ is separated into open connected components $\sP_i$'s,
and $\sP_i\cap \sP_j=\emptyset$, for $i\ne j$. \\
(3) follows from a classical result in topology, telling that ``open connected sets in $\Rn$ are
path connected''.  [It is also possible to argue directly, since, given that the $\sP_i$'s are open,
an open ball centered at any point $x\in \sP$ must intersect all coordinate directions].
\end{proof}

\begin{rem}
Properties similar to (1)-(2)-(3) above are still true in case the Hessian is hyperbolic.
Indeed, considering the set
$$\sP:=\{x\in \Rn\ : \ n_0(H(x))=0\ ,\, n_+(H(x)) \,\ \text{and}\,\ n_-(H(x))
\,\ \text{constant}\ne 0 \}\ ,$$
this set is open and the union of (path) connected components.
The reason is that perturbation of a hyperbolic Hessian renders a hyperbolic one, with
same inertia, as a consequence of the fact that
invertible matrices form an open set. 
\end{rem}

\subsection{Courant's theorem}
As we will see in the following, a main idea of our method is to escape the basin of
attraction of a minimizer by searching for a saddle point.  Now,
it is well-known that if the potential is a function of one
real variable, $g:\R\to \R$,  and $x_1$ and  $x_2$ are two strict minima,
that is $g''(x_{1,2})> 0$,
then $g$ must have another critical point $x_3$ between $x_1$ and $x_2$. 
However, as soon as we consider a real-valued function of two variables, a similar result
does not hold, in general\footnote{As an example, consider the function 
$g_1(x,y)=(x^2y-x-1)^2+(x^2-1)^2$: it has two local minima at $(1,2)$ and $(-1,0)$,
and no other critical point.}.  
Nevertheless, under certain conditions the existence of other non-extremal critical points
has been proved, and this result, due to Courant, dates back to 1950 (see \cite[p.49]{Jabri},
where $g$ is only assumed to be $\mathcal{C}^1$).
\begin{theorem}\label{Courant}
Suppose that $g$ is coercive\footnote{Recall 
$g:\R^n\rightarrow \R$ is coercive if $\lim_{\|x\|\rightarrow \infty} g(x)= +\infty$, 
that is for any constant $M$ there is a constant $R_M$ such that $\|g(x)\|>M$ 
whenever $\|x\|>R_M$.} and possesses two distinct strict relative minima $x_1$ and $x_2$. 
Then $g$ possesses a third critical point distinct from $x_1$ and $x_2$, characterized by 
$$g(x_3)=\inf_{\Sigma \in \Gamma} \max_{x\in \Sigma} g(x)\ ,$$
where $\Gamma= \{ \Sigma \subset \R^N; \Sigma \mbox{ is compact 
and connected and } x_1,x_2 \in \Sigma\}$.\\
Moreover, $x_3$ is not a relative minimizer; that is in every neighborhood of $x_3$, 
there exists a point $x$ such that $g(x)<g(x_3)$. 
\end{theorem}
Thus, $x_3$ in Theorem \ref{Courant}
can be viewed as the set which topologically separates $x_1$ and $x_2$.

Theorem \ref{Courant} is part of ``mountain pass theory.'' An
accessible introduction to this theory and its applications
is in \cite{Siam15}, a comprehensive treatment is \cite{Jabri}, and the report
\cite{MoreMunson} and the work \cite{ZhangDu} propose
numerical methods to approximate mountain pass points
(here, the authors use the characterization of mountain pass points as critical points where
the (nonsingular) Hessian has exactly one negative eigenvalue).


\section{Descent directions and the double-descent}\label{descents}
Here we introduce the double descent direction.
First, we recall the definition of (gradient) descent and Newton's directions.

\begin{itemize}
\item[(a)] ({\sl Descent direction}).  Assuming that $\grad g(x_0)\ne 0$, any
direction $v$ such that $g(x_0+\alpha v)<g(x_0)$, for some $\alpha>0$, is called
a {\sl direction of descent} for the potential $g$.  A trivial computation shows that a
direction of descent
$v$ must satisfy $(\grad g(x_0))^T v<0$.  The classic choice is $v=-\grad g(x_0)$
(the so-called {\emph{gradient descent}} choice). 

\item[(b)] ({\sl Newton's direction}).  This is the direction resulting from using
Newton's method to solve the problem $\grad g(x)=0$.  In other words, it is the direction
(assuming that $\grad g(x_0)\ne 0$ and that the Hessian is invertible) given by 
$v(x_0)=-H(x_0)^{-1}\grad g(x_0)$.  We note that this is a descent direction
for the functional $G$ at $x_0$, since $\grad G=H \grad g$.

\end{itemize}

\begin{rem}
Of course, we can always normalize a descent (and/or Newton's) direction to be
a vector of norm $1$.
\end{rem}

\subsection{Descent direction within a positive definite region}
The following result, which is both fundamental and well known (see \cite[p.114]{DS}),
serves as motivation for some of our later algorithmic choices.

\begin{lemma}\label{NewtOK}
Let $x_0$ be a point where $\grad g\ne 0$ and let $H(x_0)$ be positive definite.
Then, the Newton's direction is a direction of descent for $g$.
\end{lemma}
\begin{proof}
We need to show that --at $x_0$-- we have
$(\grad g)^T v<0$ when $v=-H^{-1}\grad g$.  But this is obvious
since $H$ is positive definite.
\end{proof}
With the help of Lemma \ref{NewtOK}, the following is immediate.
\begin{proposition}\label{Prop1}
Suppose that $x_0\in \Omega_0$, where $\Omega_0$ is 
a path-connected component where $H$ is positive definite. 
Then, either $\grad g(x_0)=0$, or
there exists a direction $v\in \Rn$, and a scalar $\alpha>0$, such that
both $G(x_0+\alpha v)< G(x_0)$ and $g(x_0+\alpha v)< g(x_0)$.  Further,
one can also choose $\tau$ so that both potentials decrease and
$H(x_0+\tau v)$ is positive definite.  
\end{proposition}
\begin{proof}
We want $v$ such that both of these relations hold at $x_0$:
$$v^T \grad G <0 \quad \text{and} \quad v^T \grad g <0\,.$$
Since $\grad G=H \grad g$, the sought relations rewrite as
$$v^T H \grad g <0 \quad \text{and} \quad v^T \grad g <0\,.$$
Each of the above inequalities defines an open half space, and the
Newton's direction is in both of these.  Therefore, the existence of a unit
vector $v$ giving us the sought decrease is established,
and there exists a scalar $\alpha$, positive, such that both
$G(x_0+\alpha v)< G(x_0)$ and $g(x_0+\alpha v)< g(x_0)$. 
\\
Further, since $H(x_0)$ is positive definite, then there is an open ball
centered at $x_0$ and of radius $r>0$, $B_r(x_0)$,
such that $H(x)$ remains positive definite for any $x \in B_r(x_0)$.  
Therefore, there exist $\tau$ so that $H(x)$ is positive definite
for $x=x_0+\tau v$.
\end{proof}

\begin{rem}\label{Rem1}
Because of Lemma \ref{NewtOK}, in Proposition \ref{Prop1} we can choose $v$
to be the Newton's direction.  
\end{rem}

We note right away that
it is often not advisable to select the step-length $\tau$ so that the Hessian remains
positive definite.  Indeed, in our numerical experiments, doing so often resulted
in a severe restriction of the step-length and inefficient computations, and it was
quite preferable to allow a decrease in the potentials without forcing a fixed inertia
for the Hessian.  For this reason, we now define the double-descent direction
allowing for the Hessian to be indefinite.

\subsection{Descent direction within an indefinite region}

Here, we generalize the above result to the case of regions
with different Hessian's inertia.

\begin{proposition}\label{main}
Let $\Omega$ be a 
path-connected region where $\nu(H)=\{n_+,n_0,n_-\}$ for
all $x\in \Omega$, with $n_+\ge 1$. Let $x_0\in \Omega$, and let
$V_+=\span\{v_1,\dots, v_{n_+}\}$ be the subspace spanned by the first $n_+$
eigenvectors of $H(x_0)$.
\\
Then, if $V_+^T\grad g(x_0)\ne 0$, 
there exists a direction $v\in \Rn$, and a scalar $\alpha>0$, such that
both $G(x_0+\alpha v)< G(x_0)$ and $g(x_0+\alpha v)< g(x_0)$.
Further, if $n_0=0$, i.e. $H(x_0)$ is hyperbolic,
then there exists $\tau>0$ so that
$\nu(H(x_0+\tau v))=\nu(H(x_0))$.
\\
Finally, in all cases above, the direction $v$ can be taken as 
$v=-(H_{+}(x_0))^{\dagger}\grad g(x_0)$,
where $H_+$ is the closest positive semidefinite matrix to $H$; that is,
if $H=V\Lambda V^T$, and $\Lambda=\diag(\Lambda_+, \Lambda_{0}, \Lambda_-)$, then
$H_+=V\Lambda^+ V^T$, with $\Lambda^+=\diag(\Lambda_+, 0_{n_0} , 0_{n_-})$, and thus
$(H_+)^\dagger=V_+(\Lambda_+)^{-1} V_+^T$.
\end{proposition}
\begin{proof}
We want $v$ such that both of these relations hold at $x_0$:
$$v^T H \grad g <0 \quad \text{and} \quad v^T \grad g <0\,.$$
Considering the direction $v$ given in the statement, we have
$$v^T\grad g=-\grad g(x_0)^T(H_+(x_0))^{\dagger}\grad g(x_0)<0$$
since $V_+^T\grad g(x_1)\ne 0$.  
For the same reason, we also have
\begin{eqnarray*}
v^T H \grad g & =-\grad g(x_0)^T(H_+(x_0))^{\dagger} H \grad g(x_0) =  \\
& = -(V_+^T\grad g)^T(\Lambda_+)^{-1}(V_+^T\grad g)<0\,.
\end{eqnarray*}
\\
Therefore, the existence of a unit
vector $v$ giving us the sought decrease is established, and
there exists a scalar $\alpha$, positive, such that both
$G(x_0+\alpha v)< G(x_0)$ and $g(x_0+\alpha v)< g(x_0)$. 
\\
Further, if $H(x_0)$ is hyperbolic, then
$\nu(H(x_0))=\{n_+,n_-\}$.  Therefore, 
there is an open ball centered at $x_0$ and of radius $r>0$, $B_r(x_0)$,
such that $\nu(H(x))=\nu(H(x_0))$ for any $x \in B_r(x_0)$.  Thus, we can
choose $\tau>0$ such that $\nu(H(x_0+\tau v))=\nu(H(x_0))$.
\end{proof}

\begin{rem}
The direction $v=-(H_+(x_0))^{\dagger}\grad g(x_0)$
of Theorem \ref{main} is effectively the Newton's direction
restricted to the subspace associated to the positive eigenvalues.
\end{rem}

\begin{summary}\label{SummMin}
To sum up, as long as the point $x_0$ is in a region where the Hessian has at least one
positive eigenvalue, and $\grad g$ has a non-trivial component in the subspace
spanned by the eigenvectors corresponding to the positive eigenvalues, we can always find a
direction $v$ which is of descent for both $G$ and $g$. If $H(x_0)$ has no $0$ eigenvalue,
we can also maintain
the inertia of $H$ by taking a step in the direction $v$; however, this may be counterproductive
(as our computations showed), since it may unduly restrict the step $\tau$, and 
it is much more desirable
to let the iterate enter and exit regions of different Hessian's inertia while
decreasing the potentials $g$ and $G$.
\end{summary}
\begin{rem}
One more comment is needed about the condition  $V_+^T(x)\grad g(x)\ne 0$. \\
The dimension of the subspace $\span(V_+)$ is $n_+$, while $\grad g$ is a
vector in $\Rn$.  Therefore, in general, the requirement  $V_+^T\grad g= 0$
would define a set of $n_+$ equations in the $n$ variables $x\in \Rn$.  Generally,
these define an $n-n_+$ dimensional manifold immersed in $\Rn$.
Therefore, we should expect that, at any $x$,
the vector $\grad g$ will have a nontrivial component in $V_+$.
This is the truer the larger is $n_+$.  
In the limiting case of $n_+=n$,
Lemma \ref{NewtOK} already told us that.  At the same time,
if $n_+=0$, then obviously there is no direction $v_+$ to begin with.
In this case, there is no double descent direction to begin with, and
our method (see below) will revert to using the gradient direction.
\end{rem}

\section{\sffamily{Double-descent colored-diffusion method}}\label{DD-CID}
The main idea of our technique is to take advantage of the knowledge of the
Hessian inertia, in order to explore the landscape going from a saddle point to 
a minimum and vice versa, avoiding being trapped in the basins of attraction 
of the critical points while following an educated path.

There are two
types of processes we use: ``local zoom-in,'' and ``basin escaping'' methods.\\
\subsection{Reaching  a critical point: local search}
This part is based on the developments of Section \ref{descents}.
We distinguish between the cases of searching for a minimum or a saddle.  In the
former case, we have a host of possibilities: 
the double descent algorithm, gradient-descent, Newton's method (damped);
in the latter, we use a (damped)
Newton's method.  Still, we must make some careful implementation choices.
\\
For example, when searching for a minimum, beside the usual concerns on how
to choose the step-length (see \cite{DS}), we also accounted for these aspects
when implementing the double descent method.
\begin{itemize}
\item[(i)]
When using the double-descent direction, we demand that
both $g$ and $G$ have appreciably decreased.
\item[(ii)]
To declare that $\grad g$ has no meaningful component in the
subspace spanned by $V_+$, we used the criterion
$\|V_+^T(x_k)\grad g(x_k)/\|\grad g(x_k)\| \|\le \sqrt{n}/10$.
When this happens, we reverted to the
direction $v$ given by the simpler gradient-descent direction, and
kept this descent strategy for 5 steps before retrying the double descent
direction.  Likewise, we reverted to the gradient descent direction
when too many damping steps are taken with the double descent
direction.
\item[(iii)]
Another important consideration is about the stopping criterion
(both when searching for a minimum or a saddle). 
In our implementation we chose
the following stopping criterion (always within the maximum allowed number
of iterations).  We iterate as long as:\\
\centerline{$\|\grad g(x_k)\|\ge {\mathtt{atol}}+ \|\grad g(x_0)\| {\mathtt{rtol}} \,\
\text{and} \,\ \| x_k-x_{k-1} \|\ge {\mathtt{atol}}+ \| x_0\| {\mathtt{rtol}}\ .$}
\end{itemize}

\subsection{Basin escaping by Color Diffusion}\label{BECD}
Here we adopt (some of) the ideas in Section \ref{IntDiff} in order to
leave the basin of attraction (for Newton's method) of a critical point.
In particular, compare 
\eqref{CIDmin-to-sad}, \eqref{CIDsad-to-min-1}, and \eqref{CIDsad-to-min-2},
with \eqref{IDsdeDiscrete}; naturally, just as \eqref{IDsdeDiscrete} can be viewed as
a discretization of the SDE \eqref{IDsde}, also our equations \eqref{CIDmin-to-sad}, 
\eqref{CIDsad-to-min-1}, and \eqref{CIDsad-to-min-2}, can be thought as
discretizations of an underlying SDE, in regions where the
Hessian inertia and the dominant (respectively, smallest) eigenvalue
are not changing.

As seen in Section \ref{descents}, the double-descent method 
is designed to lead to a minimum when the Hessian at the starting point
is either positive definite or indefinite, and -when properly implemented-
it will really be Newton's method close to a minimum. 
Instead, when the initial value lies in a 
region in which both $n_{+}(H)$ and $n_{-}(H)$ are nonzero, we will
presume that (damped) 
Newton method will converge to a saddle point (or perhaps to a maximum);
this expectation has
effectively been borne out in practice for the vast majority of
our experiments. 

Regardless, if we are at either a minimum or at a saddle, we need to leave
the respective basins of attraction for (damped) Newton method.  To do this,
we implemented a colored intermittent
diffusion method as follows, by reflecting the choices made above to look
for either a saddle or a minimum, and the discussion in Section \ref{IntDiff}.

\begin{itemize}
\item[(a)] {\emph{From a min $x_0$, trying to go to a saddle}}.  
Let us first assume that, $\lambda_1>\lambda_2\ge \cdots \ge \lambda_n$ along
our iterates.
There are three basic steps.
\begin{itemize}
\item[(i)]
Select $\alpha$ (usually, $\alpha=1$) and generate
$$x_1=x_0+\alpha v_1(x_0)v_1^T(x_0)W\ ,\,\ W\in N(0,1)^n\ .$$
\item[(ii)]
Find $h$ such that $|G(\hat x_{k+1})|<|G(x_k)|$, with $\hat x_{k+1}\  = \ 
x_k-h\bigl(H^\dagger(x_k)\grad g(x_k)\bigr)$, and then
\begin{equation}\label{CIDmin-to-sad}\begin{split}
x_{k+1}\ & = \ x_k-h\bigl(H^\dagger(x_k)\grad g(x_k)\bigr)+\alpha \sqrt{h} \sigma(x_k)W \ ,
\text{where}\\
\sigma(x_k) \ & = \ -v_1(x_k)v_1(x_k)^T\ ,\quad \text{and} \quad W\in N(0,1)^n\ .
\end{split}
\end{equation}
\item[(iii)]
Continue with the diffused damped Newton's above until the
Hessian has some negative eigenvalues or the maximum 
number of diffusive iteration has been exceeded.  At that point, use
(damped) Newton's method.  Hence, select the Newton direction $v$
and the step length $h$ to decrease the auxiliary potential $G$; say,
$G(x_k+hv)<G(x_k)$.  If $g_m$ denotes the value of the potential $g$ at
the minimum from which we started, we observed that consistently
$g(x_k+hv)>g_m$ which betrays that we are not going back to the starting
minimum.
\end{itemize}
The rationale for the colored noise diffusive step is to move away as quickly as possible
from the basin of attraction of the minimum.  If $x_0$ is a minimum, the standard
quadratic approximation in a $\epsilon$-ball around $x_0$ will give:
$$g(x_0+\epsilon y)=g(x_0)+\epsilon \grad g(x_0)^Ty + \frac12 \epsilon^2 y^TH(x_0)y+ \dots $$
and therefore, with $\|y\|=1$, the fastest increase is for $y=v_1$.
In the (very unlikely) case that
the dominant eigenvalue has multiplicity greater than $1$, we select a random
vector in the span of the dominant eigenvectors.

\item[(b)] {\emph{From a saddle $x_0$, trying to go to a min}}.
Let us first assume that $\lambda_1\ge \cdots \ge \lambda_{n-1}> \lambda_n$.
Even here there are two basic steps, getting out of the saddle and going
to a minimum.  The second step, see below,
can be carried out with the double descent method, or with gradient descent,
or possibly with a (damped) Newton
approach.  In all cases, first we use colored diffusion 
steps to move out of the saddle.
\begin{itemize}
\item[(i)]
Select $\alpha$ (usually, $\alpha=1$) and generate
$$x_1=x_0+\alpha v_n(x_0)v_n^T(x_0)W\ ,\,\ W\in N(0,1)^n\ .$$
As before, the choice of the colored noise diffusive step is to move away as quickly as possible
from the basin of attraction of the saddle, while decreasing the potential $g$.
If $x_0$ is a saddle, in a $\epsilon$-ball around $x_0$ we have:
$$g(x_0+\epsilon y)=g(x_0)+\epsilon \grad g(x_0)^Ty +\frac12 \epsilon^2 y^TH(x_0)y+ \dots $$
and therefore, with $\|y\|=1$, the fastest decrease is for $y=v_n$.
\item[(ii)]
{\bf Double descent}.
If $\grad g(x_k)$ has a meaningful
component in the direction of $V_+(x_k)$, do (i-a), otherwise do (i-b).
\begin{itemize}
\item[(ii-a)]
Find $h$ such that both $|G(\hat x_{k+1})|<|G(x_k)|$, and 
$|g(\hat x_{k+1})|<|g(x_k)|$, with $\hat x_{k+1}\ = \ x_k-h\bigl(H_+^\dagger(x_k)\grad g(x_k)\bigr)$, and then
\begin{equation}\label{CIDsad-to-min-1}\begin{split}
x_{k+1}\ & = \ x_k-h\bigl(H_+^\dagger(x_k)\grad g(x_k)\bigr)+\alpha \sqrt{h} \sigma(x_k)W\
\ ,\text{where}\\
\sigma(x_k) \ & = \ -v_{n}(x_k)v_{n}(x_k)^T\ ,\quad \text{and} \quad W\in N(0,1)^n\ .
\end{split}
\end{equation}
\item[(ii-b)] Find $h$ such that $|g(\hat x_{k+1})|<|g(x_k)|$, with 
$ \hat x_{k+1}\ = \ x_k-h\ \grad g(x_k)$,  and then
\begin{equation}\label{CIDsad-to-min-2}\begin{split}
x_{k+1}\ &= \ x_k-h\ \grad g(x_k) +\alpha \sqrt{h} \sigma(x_k)W \ ,\text{where}\\
\sigma(x_k) \ & = \ -v_{n}(x_k)v_{n}(x_k)^T\ ,\quad \text{and} \quad W\in N(0,1)^n\ .
\end{split}
\end{equation}
\end{itemize}
\item[(iii)]
Continue with the diffused Double Descent above until the
Hessian has all positive eigenvalues or the maximum 
number of diffusive iteration has been exceeded.
At that point, use
Double Descent method.  Hence, select the direction $v$
and the step length $h$ to decrease the potential $g$ auxiliary potential $G$; say,
$g(x_k+hv)<g(x_k)$ and $G(x_k+hv)<G(x_k)$.
\end{itemize}
Again, in the (very unlikely) event that the smallest eigenvalue has multiplicity greater than $1$,
we select a random unit vector in the corresponding subspace.
\end{itemize}

\section{\sffamily{The method at a glance}}\label{glance}

In the previous sections, we presented only the two key components of the method, 
namely, the local search and the basin escaping.
Here, we give a broader idea of the method.

\begin{itemize}
\item[(0)]
The very first initial datum $x_0$ is randomly chosen (within a region of interest).
A local search for a minimum starts with the double descent method, and the point
found is stored in a Table of critical points.
\item[(1)]
A point from the Table is randomly selected.  Colored diffusion to escape
the basin of this critical point is performed (see Figure \ref{fig:1}), followed by a local search
for the next critical point.  The new point is stored in the Table\footnote{The
same point can thus appear multiple times in the Table.}, and
step (1) is thus repeated until a predefined number of critical points
is found.
\end{itemize}
\begin{figure}[hbt]
\begin{subfigure}{0.45\textwidth}
\includegraphics[scale=0.36]{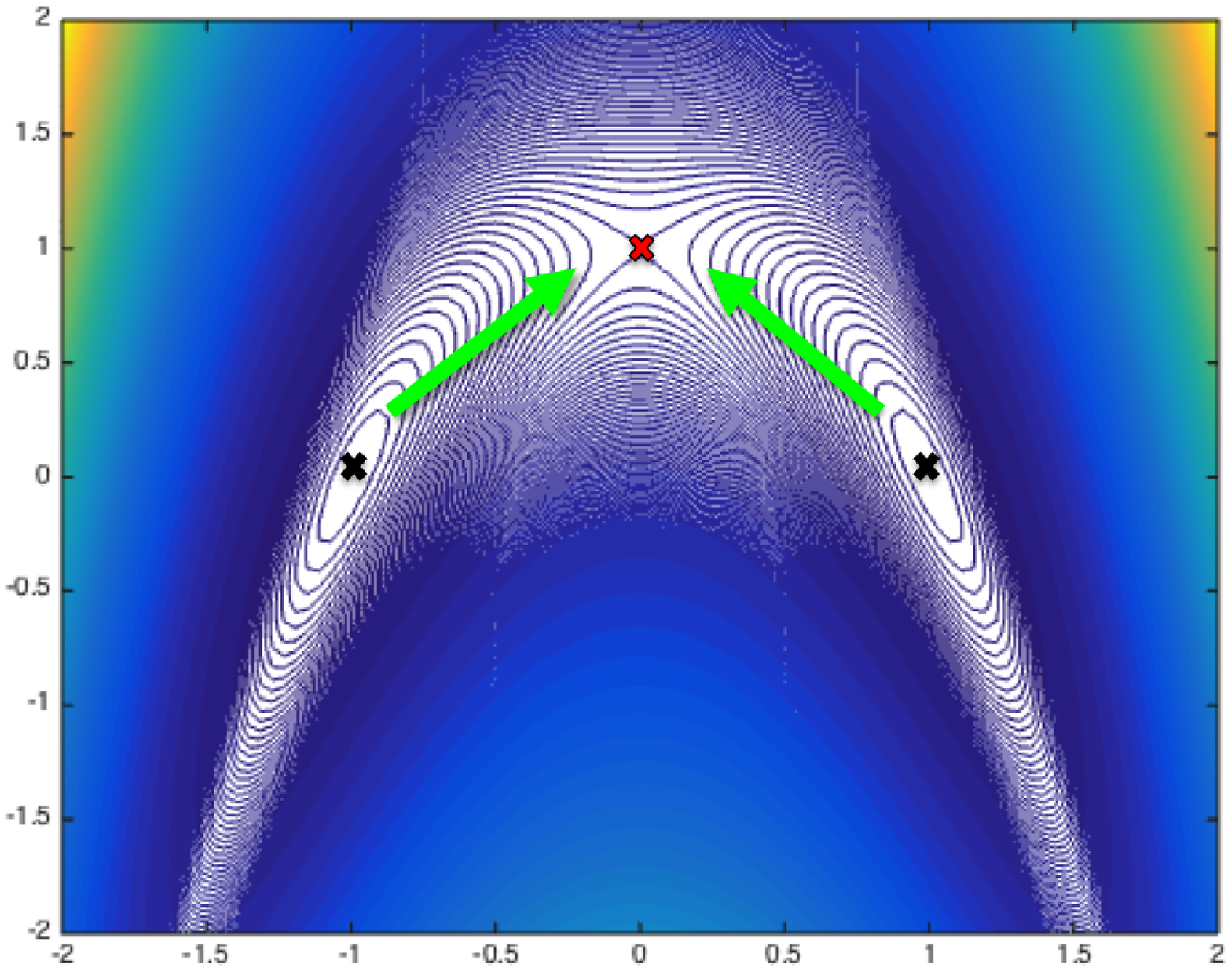}
\vspace{-3cm}
\caption{} \label{fig:1a}
\end{subfigure}
\begin{subfigure}{0.45\textwidth}
\includegraphics[scale=0.36]{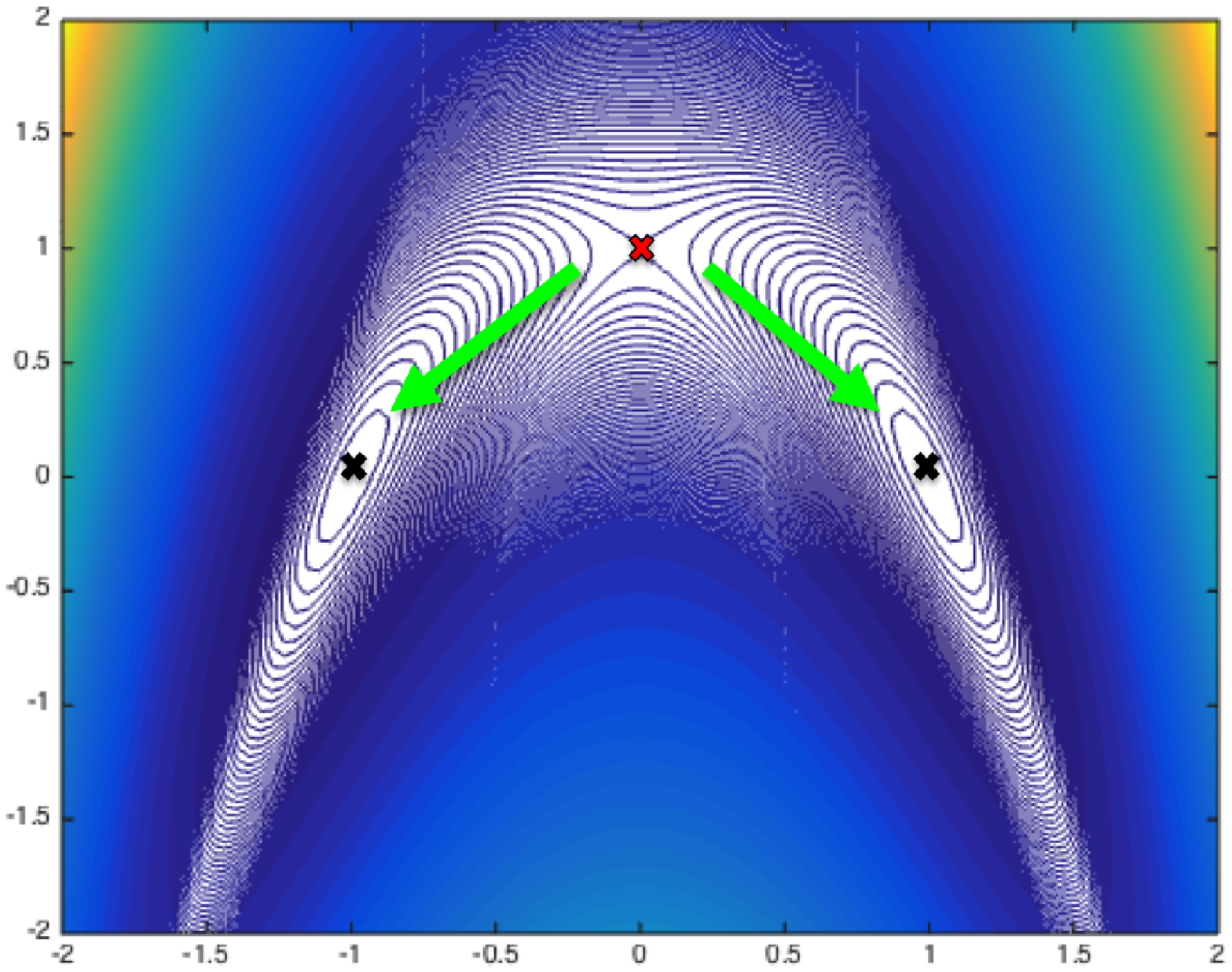}
\vspace{-3cm}
\caption{} \label{fig:1b}
\end{subfigure}
\vspace{-0.5cm}
\caption{\small{Basin escaping main idea. Consider the function of Example \ref{Ex:Molei}, 
which has two minima and a saddle point. To escape the basin of attraction of a minimum, 
the idea is to move towards a saddle point, as in Figure \ref{fig:1a}. Starting from a 
saddle point, instead,  the goal is to find a direction to move as quick as possible to a
 minimum, as illustrated in Figure \ref{fig:1b}.}}\label{fig:1}
\end{figure}
\subsection{Sketch of the algorithm}
\begin{enumerate}
\item Choose a random point $x_0$ in the search region.
\item Look for a minimum $x^{min}$ by using the double-descent method. 
\item Store $x^{min}$ in the list of critical points.
\item LOOP BEGINS - to be repeated for a preassigned number of iterations.
\begin{enumerate}
\item Randomly choose a point $x$ from the list.
\item If $x$ is a minimum, start diffusion according to \eqref{CIDmin-to-sad} until 
$n_-(H)\neq 0$ (see Figure \ref{fig:2a}) or maximim number of diffusive steps is exceeded.
\\
Apply (damped) Newton method to find a saddle point. \\
Store the saddle in the list of critical points.
\item If $x$ is a saddle point, take diffusive steps according to \eqref{CIDsad-to-min-1} 
(or \eqref{CIDsad-to-min-2}), until $n_-(H)=0$ (see Figure \ref{fig:2b}) or the maximum number 
of diffusive steps is exceeded.
Apply double descent to find a minimizer.\\
Store the minimum in the list of critical points.
\end{enumerate}
LOOP ENDS
\end{enumerate}
\begin{figure}[hbt]
\begin{subfigure}{0.45\textwidth}
\includegraphics[scale=0.35]{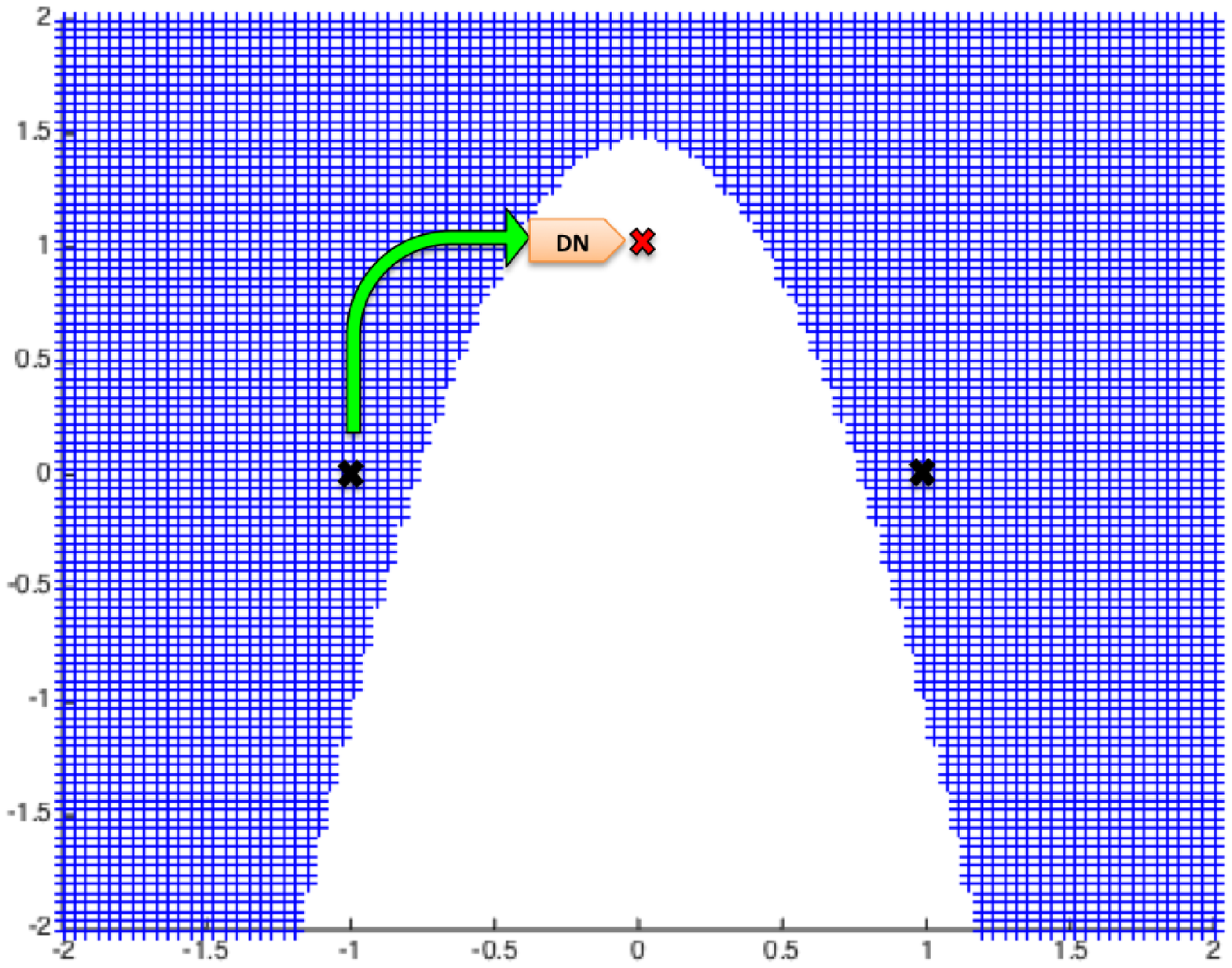}
\vspace{-3cm}
\caption{} \label{fig:2a}
\end{subfigure}
\begin{subfigure}{0.45\textwidth}
\includegraphics[scale=0.35]{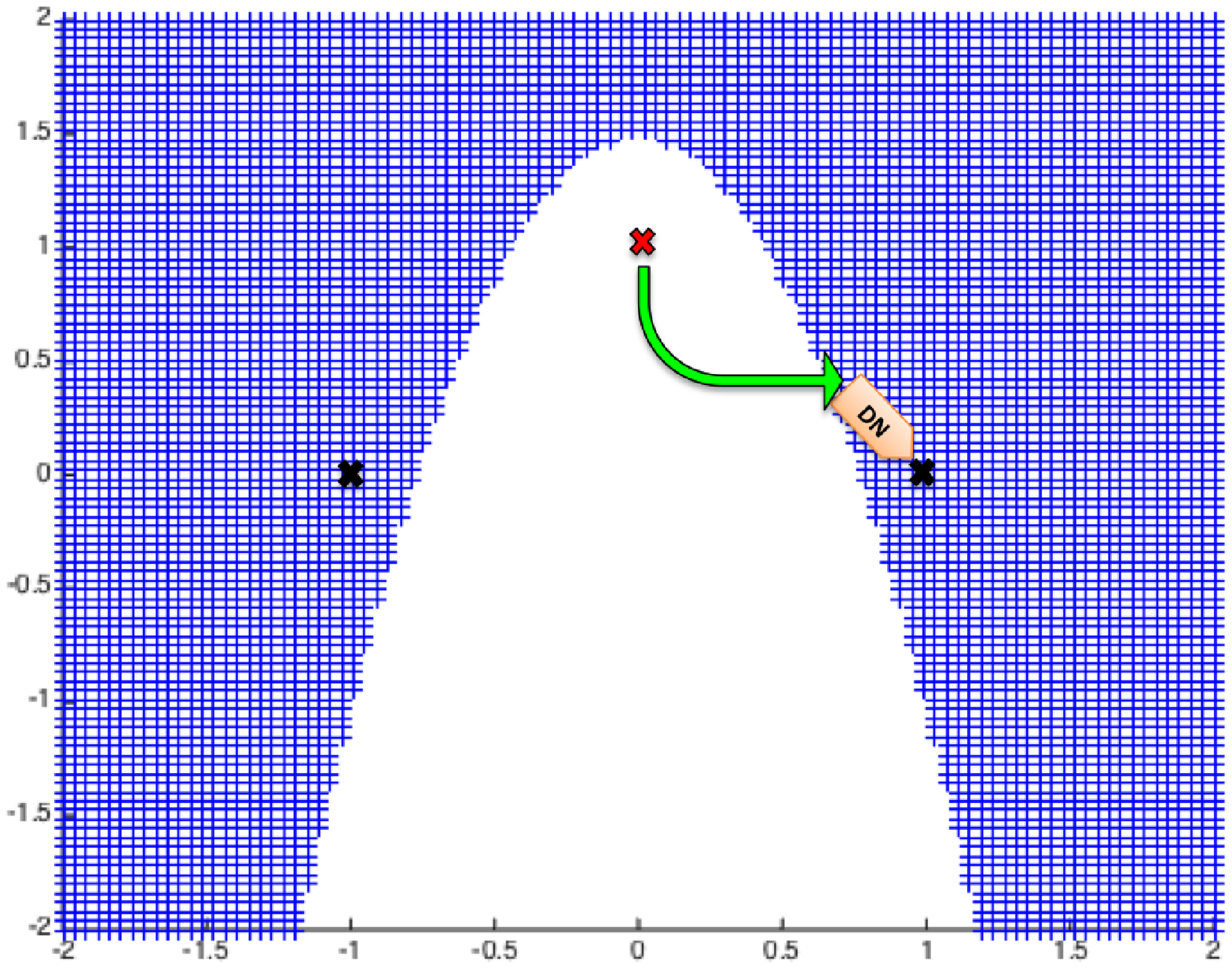}
\vspace{-3cm}
\caption{} \label{fig:2b}
\end{subfigure}
\caption{\small{Switch timing for colored diffusion. The blue region represents 
the set of points at which the Hessian is positive definite, while the white region 
has an indefinite Hessian. Suppose we need to escape the basin of attraction of a 
minimum (Figure \ref{fig:2a}), the diffusion process is triggered until reaching the 
white region, where a local search starts. Conversely, starting from a saddle point, 
as in Figure \ref{fig:2b}, the diffusion process stops when the blue region is reached, 
and the local search begins.}}\label{fig:2}
\end{figure}

\subsection{Computational considerations}
A main drawback of Newton's type technique, hence also of the
double descent method, is the need to form, evaluate, and decompose, the Hessian.
Except for problems where the Hessian is simple to evaluate, and very structured
(e.g., tridiagonal), this can be very expensive and it restricts applicability
of Hessian based techniques to small dimension (say, up to a few hundred
variables on a typical laptop).  We also note that for some problems 
evaluating the Hessian is itself an expensive task; e.g., this is the case for
interatomic potentials, such as the Morse and Lennard-Jones potentials (see Section
\ref{Expers}).  Although our purpose in this work has not been to deal specifically with
efficient implementations, but rather to give ways to
explore the phase space (the landscape), for large (possibly sparse)
problems we have experimented with Lanczos techniques, 
and a subspace version of Newton's method,
whereby we project the Hessian in the direction of the most dominant eigenvalues
(positive and negative).  We will report on these apects in
other works.

An important consideration pertains to the colored noise diffusion.
To perform this diffusion, and to monitor when to stop it, it
is straightforward to bypass the Hessian factorization.
In fact, to form the color noise, and to decide when to stop diffusion,
we only need the two eigen-directions $v_1$ and $v_n$.
These are inexpensive to obtain with a well designed Lanczos technique (e.g.,
{\tt eigs} in {\tt Matlab}), by asking for
(respectively) the largest and smallest eigenvalues.  
This feature is particularly
useful when using just a damped Newton's method with color noise (as in our basic
intermittent diffusion method from a min to a saddle),
since the linear systems arising during the
iteration are then solved without resorting to a full eigen-decomposition.
To elucidate, and to account for the possibility of singular Hessian, we
first form the QR factorization with column pivoting of the Hessian:
$HP=QR$ with diagonal of $R$ ordered in decreasing magnitude. Then solve the
resulting triangular system, possibly for the minimum norm solution (if the
Hessian was singular, which is betrayed by $R_{nn}=0$).

Finally, as seen in Section \ref{BECD}, we use a variable stepsize controlled through
the requirement of moving in the descent direction(s). The initial stepsize
is set to $1$, and the stepsize is always required to remain in $[2^{-26},2^5]$,
where $2^{-26}= \sqrt{\tt{eps}}$, the square root of the machine precision. 
When one step is taken in the desired
direction, and the computation is immediately accepted, then the stepsize is doubled;
if the computation is rejected, the stesize is halved. 
If we reach the minimum allowed stepsize, the algorithm halts and restarts from
a different critical point in the Table (or a different random point, the very
first time).

\section{Applications and examples}\label{Expers}
In this section, we show performance of our method on several problems,
both standard model problems, with an illustrative purpose and 
to validate the method on different landscape features, and those arising from
chemical potentials. 

\subsection{Test Problems}
 \subsubsection{An illustrative example}\label{Ex:Molei}
 Consider the following elementary potential:
  \begin{equation}\label{Molei} g(x,y)= (x^2-1)^2+(x^2+y-1)^2.
  \end{equation}
  It has 2 minimizers, located at $(-1,0)$ and $(1,0)$, and a saddle point at $(0,1)$.
Starting from a random point $x_0$, 
the Double Descent technique quickly leads to a minimizer, 
and the diffusion combined with Newton method allows to find the saddle point,
from which the algorithm looks again for a minimizer.
By using our technique, and asking the algorithm for at most 4 critical points,
we were able to find, in a single run, the two minimizers and the saddle point. 
Indeed, the method is behaving exactly like we were hoping:
first, it converges to $(-1,0)$, then it goes through the saddle point and 
from there localizes the other minimum at $(1,0)$, and then it goes back to the saddle point.
On average we counted 3 diffusive steps and $8.25$ local search iterations. 

  
\subsubsection{Shubert Function}
The Shubert function is a highly multimodal potential: it has several local minima and 
many global ones. Naturally, the function presents many saddle points and maxima. 
Moreover, the global minima and the global maxima are extremely close, and this is one of 
the reason why it may be difficult to find the global minimizers.\\
\begin{figure}[hbt]
\begin{subfigure}{0.45\textwidth}
\includegraphics[scale=0.39]{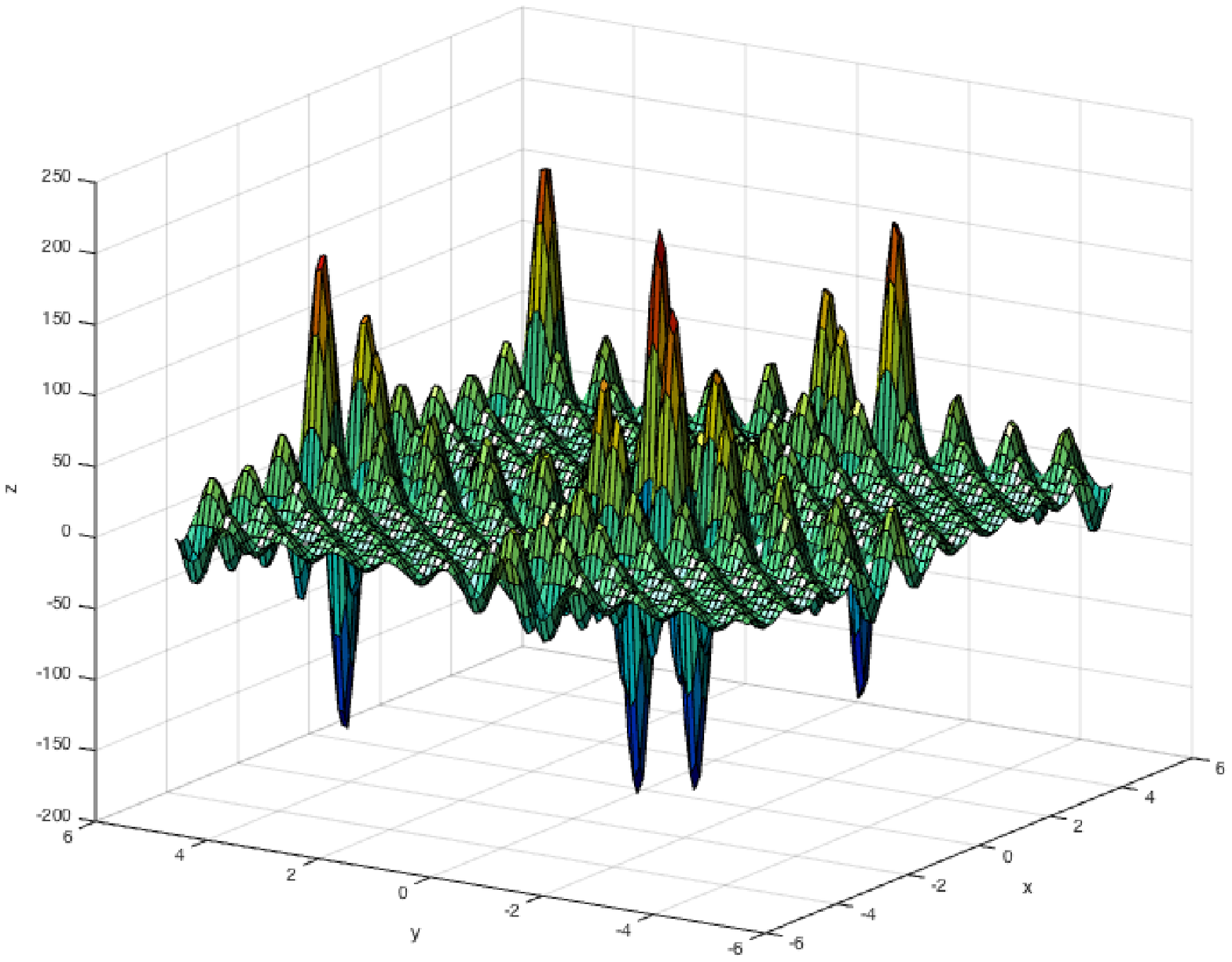}
\vspace{-1cm}
\caption{} \label{fig:3a}
\end{subfigure}
\begin{subfigure}{0.45\textwidth}
\includegraphics[scale=0.4]{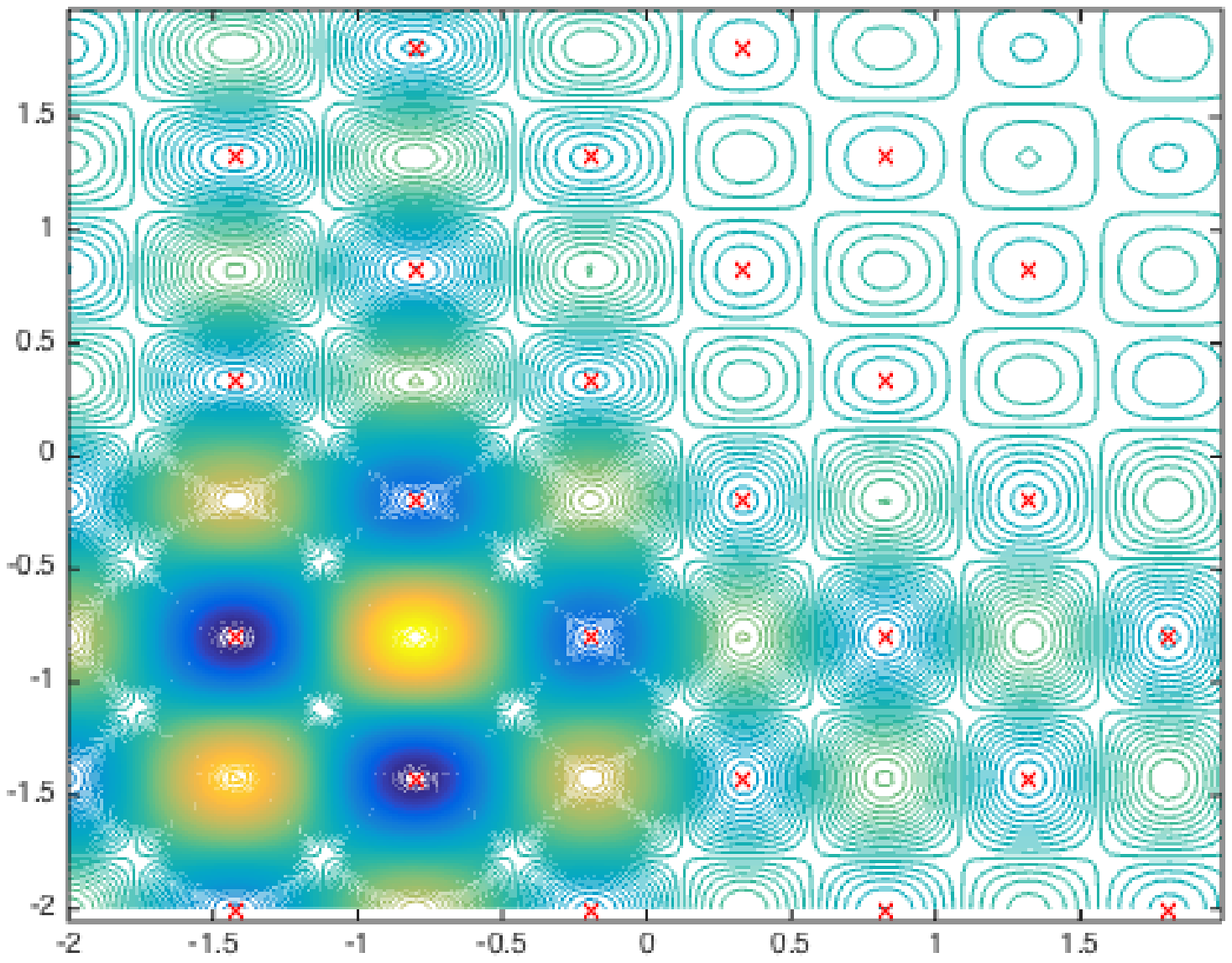}
\vspace{-1cm}
\caption{} \label{fig:3b}
\end{subfigure}
\caption{\small{Shubert Function}}\label{fig:3}
\end{figure}
Although our algorithm is designed to find minima and saddles, it could end 
up finding maxima as well, due to the Newton's basins of attractions, which are nontrivial.
Schubert's potential,
\begin{equation}\label{Shubert}
g(x,y)=\left(\sum_{i=1}^5 i\, \cos\big[ (i+1)\,x+  i\big]\right)
\left(\sum_{i=1}^5 i\, \cos\big[ (i+1)\,y+  i\big]\right),
\end{equation}
is represented in Figure \ref{fig:3a}. Figure \ref{fig:3b} is a zoom of 
the contour plot around the global minimizer, and the points are the minimizers 
found by just applying the double descent technique, starting from random initial values. 
While finding the global minimizer by applying a deterministic technique requires a starting 
point in its neighborhood, the ability to explore the landscape eliminates this necessity. 
One single run of our technique, asking for 100 possible critical points,
gave us 45 minima, 3 of which were global, at different locations.
On average, for attempt, we counted one diffusive step and 5.9 local search iterations.

\subsubsection{Biggs Function}
Let us consider the following function:
   \begin{equation}\label{Biggs2} g(X)= \sum_{i=1}^{10} 
   \left(e^{-t_i\, x_1}-5e^{-t_i\, x_2} -y_i\right)^2\end{equation}
 where $t_i=0.1\, i$, $y_i= e^{-t_i}-5e^{10\,t_i}$. \\
There are two critical points: a minimum 
at $(1,10)$ and a saddle at $(16.7047,16.7047)$, 
as shown in figure \ref{fig:Biggs_contour}.\\
\begin{figure}[hbt]
  \includegraphics[scale=0.4]{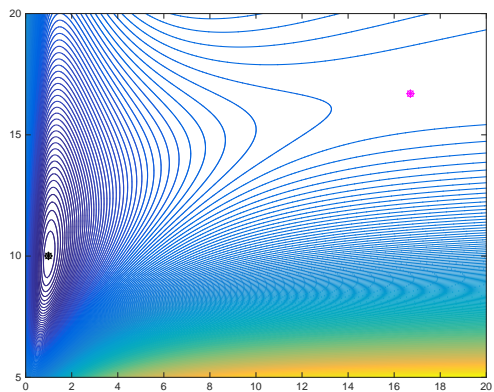}
\caption{\small{Biggs function's critical points: (global) minimum  in black and saddle point in magenta.  }}
\label{fig:Biggs_contour}
\end{figure}

The challenges in this problem are the flat landscape of the potential and the 
presence of narrow regions in which the Hessian is positive (negative) definite, but that do not 
contain a minimum (maximum), as shown in Figure \ref{fig:Biggs2H}.
\begin{figure}[hbt]
\includegraphics[scale=0.4]{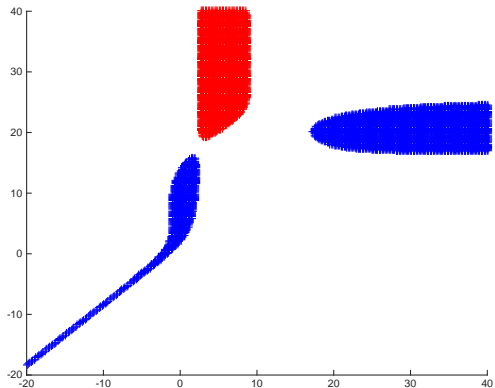}
\caption{\small{Biggs function: in blue the regions of the plane in which the Hessian is positive-definite, 
in red the one in which $H$ is negative-definite; elsewhere $H$ is indefinite. }}
\label{fig:Biggs2H}
\end{figure}
Asking for 2 critical points, we were able to find, in a single run,
both the minimimum and the saddle point. 
On average, we performed 2.5 diffusive steps and $18$ local search iterations.

\subsubsection{Camel Function}
Let us consider the following function:
\begin{equation}\label{Camel}
g(x,y)=\left(4-2.1\ x^2+\frac{x^4}{3}\right)x^2+xy+4(y^2-1)y^2
\end{equation}
This function is a standard test function for global optimization, 
but it is also useful as a test for the mountain passes' search (see  \cite{MoreMunson}).
In fact, our algorithm can be also used to compute mountain passes.  
As in \cite{MoreMunson}, these are characterized as 
critical points that satisfy the Palais-Smale 
condition and whose Hessian has exactly one negative eigenvalue, that is $n_-(H)=1$.

Consider the region $[-2,2]\times[-1,1]$.  Here, there are $14$ critical points: 
$6$ mountain passes, $6$ minima and $2$ maxima.  Our method has no difficulty
in computing all of these points in one single execution.
Results are summarized in Figure \ref{fig:Camel} and in Table \ref{tab:Camel}.

\begin{figure}[hbt]
\includegraphics[scale=0.45]{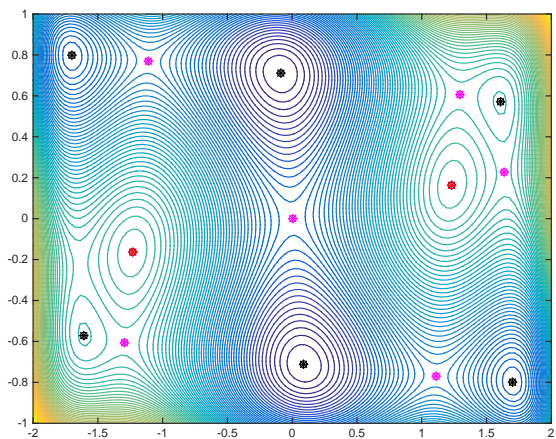}
\caption{\small{Camel Function's critical points. In the region of interest, represented in this figure, 
there are 6 minima (in black), 2 maxima (in red), and 6 mountain passes (in magenta).}}
\label{fig:Camel}
\end{figure}
   \begin{table}[htbp] 
\begin{tabular}{||*{10}{c||}|} 
\hline
 $x$ & $y$ & $f(x,y)$ & spectrum of H \\
 \hline
   0.0898  &  -0.7127 &  -1.0316  &  7.6822 \;  16.4932\\
   -0.0898  &  0.7127  & -1.0316  &  7.6823 \;  16.4932\\
      1.6071  &  0.5687 &   2.1043  &  7.1215  \; 10.0216\\
       -1.6071  & -0.5687 &   2.1043  &  7.1215  \; 10.0216\\
    1.7036 & -0.7961  & -0.2155  & 18.8171 \;  22.6975\\
     -1.7036  &  0.7961 &  -0.2155  & 18.8171 \;  22.6975\\
\hline
    1.2302  &  0.1623  &  2.4963  & -8.0149  \; -5.9537\\
       -1.2302  & -0.1623  &  2.4963  & -8.0149  \; -5.9537\\
       \hline
         0        &    0          &   0           &  -8.0623 \;   8.0623\\
            1.1092  & -0.7683  &  0.5437  & -7.9026  \; 20.3667\\
              -1.1092  &  0.7683  &  0.5437 &  -7.9026 \;  20.3667\\
   1.2961  &  0.6051  &  2.2295  & -6.1772  \;  9.6376\\
 -1.2961  & -0.6051  &  2.2295  & -6.1772  \;  9.6376\\
    1.6381  &  0.2287  &  2.2294  & -5.5458  \; 12.4367\\
   \hline
\end{tabular}
\caption{Camel critical points} \label{tab:Camel} 
\end{table}

\subsubsection{Rosenbrock Function}
This function 
is given by
\begin{equation}\label{Rosen}
g(x)=\sum_{i=1}^{N-1} \left[100\,(x_{i+1}-x_i^2)^2+(x_i-1)^2\right]\ ,
\end{equation}
and it has a global minimum value of $0$ at at $(1,1,...,1)$, for any value of $N$.
The Hessian is very inexpensive to compute and factor, being
tridiagonal.
The global minimum lies inside a long, narrow, parabolic shaped flat valley; while 
finding the valley is trivial, detecting the global minimizer is not.\\

We take $N=50$.  A Monte-Carlo gradient-descent technique using $20$ random initial guesses
did not find the global minimizer.  A single implementation 
of our technique with the possibility to find at most $20$ critical points, found
$16$ minimizers, $9$ of which were global. On average, per attempt to find a 
critical point, we counted $1.8$ diffusive steps and $136$ local search steps. 
For comparison sake, the {\tt Matlab} routine {\tt GlobalSearch} found
the global minimizer, whereas the {\tt Matlab} simulated annealing routine 
{\tt simulannealbnd} gave a best value for the minimum of $40.7188$.

\subsection{Chemical Potentials}
An interesting application of global optimization is protein 
folding.  Mathematically, this consists in 
finding the equilibrium configuration of $d$ atoms, assuming that
the forces between the atoms are known.  In the end, one has to find  
the minimizer of a potential energy function depending on $3d$ variables.\\
The Lennard-Jones and Morse clusters are two well-known systems of this kind and they 
have been extensively studied, and the minima tabulated.  
For example, the (currently best) global minima
for Lennard-Jones and Morse potential can be found at the database
\cite{Database}.  These results were obtained with the methods presented in
\cite{DLLS} and \cite{LennJ}. Both are ``basin-hopping'' techniques; they 
exploit the funneling structure of the potentials (that is, the global minimizer lies at the bottom
of a monotonically descending sequence of minimizers), and make a number of choices
explicitly based on the specific problem at hand.  For example, the authors perform a continuation
based upon an optimal configuration reached with $d$ atoms to initiate the search for
$(d+1)$ atoms.\\
With no pretense of comparing to these other results,
below we present some of the results we obtained applying our general technique on
both Morse and Lennard-Jones potentials.  These
potentials depend on the mutual distances (in $\R^3$) between the atoms, namely 
$r_{ij}=  \|P_j-P_i\|$, with $1\leq i<j\leq d$ and $P_k=(x_k,y_k,z_k)$, for all $k=1,...,d$.\\
Given obvious symmetries in the problem, 
we imposed the following location constraints:
we fix one atom at the origin ($P_1=(0,0,0)$), another one on the  
$x$ axis ($P_2=(x_2, 0, 0)$), and a third one in the $xy$-plane ($P_3=(x_3, y_3,0)$). 
All other atoms are unconstrained. 
With this setup, each configuration will be identified by the vector of coordinates
$$X= (x_2, x_3, y_3, x_4, y_4, z_4, \ldots, x_d, y_d, z_d) ,$$
and the dimension of the problem becomes $N=3d-6$.\\
In our experiments we compute the gradient analytically, and the Hessian numerically, 
by forward finite differences.

\subsubsection{Lennard-Jones Potential}
This is defined as follows:
\begin{equation}\label{LJ}
V(r)= 4\,\varepsilon \sum_{i<j}^d \,\left[ \left(\frac{\sigma}{r_{ij}}\right)^{12}- 
\left(\frac{\sigma}{r_{ij}}\right)^{6} \right],
\end{equation}
where $\varepsilon$ and $2^{1/6}\sigma$ are the pair equilibrium well depth and 
separation respectively. We take $\varepsilon=\sigma=1$.\\
This is a problem where a simple gradient descent technique, 
coupled with a Monte-Carlo randomization, performs reasonably well; as a matter
of fact, our own double-descent method quite often automatically reverts to
gradient descent.  That said, 
there are specific configurations where the problem is much more difficult with our
method than with a simple gradient descent (for example, $d=38$, a value where
the optimal configuration does not have a regular
(Mackay) hycosahedral structure).  
On this problem, the basin-hopping techniques of \cite{LennJ} is a more effective
way to find the global minima, since the knowledge of the potential landscape is
exploited in the algorithm itself; our method is really a landscape exploration
approach, and our method failed, for example, for $d=38$.  Nevertheless, the method
worked well for smaller values of $d$, as reported in Table \ref{tab:LJ}.

\begin{table}[htbp] 
\begin{tabular}{||*{10}{c||}|} 
\hline
& & \\
 $d$ & N & global minima \\
 &  &\\
 \hline
  2 & 1 & -1 \\
 \hline
    3 & 3 & -3  \\
 \hline
   4 & 6 & -6 \\
 \hline
   5& 9& -9.104 \\
 \hline
 6  & 12 & -12.712	\\
  \hline
  7  &  15& -16.505\\
   \hline
  8 & 18& -19.821	\\
   \hline
   9 & 21 & -24.113\\
   \hline
   10 & 24& -28.422\\
   \hline
   11 & 27 & -32.766\\
   \hline
   12 &30  &-37.968\\
   \hline
   13 &33 &-44.327\\
   \hline
   14 & 36 & -47.845\\	
 \hline
\end{tabular}
\caption{Lennard Jones global minima} \label{tab:LJ} 
\end{table}

\subsubsection{Morse Potential}
The Morse potential 
is defined as follows:
\begin{equation}\label{Morse}
V_{\rho}(r)= \sum_{i<j}^d \,\left[e^{\rho(1-r_{ij})}\left( e^{\rho(1-r_{ij})}-2\right)\right]
\end{equation}
where $\rho$ is a parameter which determines the width of the well. 
We treat this problem as truly unconstrained, and this may create difficulties to descent techniques,
since descent directions may well identify points ``at infinity'' (i.e., some coordinates grow
unbounded); e.g., this happens to the {\tt Matlab} code {\tt GlobalSearch}.
A further difficulty is that global minima become harder to locate when $\rho$ increases. 
In Table \ref{tab:Morse}, we report the results of our method for $11$ atoms and 
$\rho=3,6,10,14$; our minima match those of \cite{Database}.
\begin{table}[htbp] 
\begin{tabular}{||*{10}{c||}|} 
\hline
 $\rho$ & global minima \\
 \hline
  3 & -37.930817 \\
 \hline
    6 & -31.521880  \\
 \hline
   10 & -30.265230 \\
 \hline
 14  &  -29.596054\\
 \hline
\end{tabular}
\caption{Morse global minima ($d=11$, $N=27$)} \label{tab:Morse} 
\end{table}
For the sake of comparison, we remark that neither {\tt Matlab} function {\tt GlobalSearch} nor
{\tt simulannealbnd} gave acceptable results.  Namely, they produced the results in Table
\ref{tab:matlab}.
\begin{table}[htbp] 
\begin{tabular}{||*{10}{c||}|} 
\hline
$\rho$ & {\tt GlobalSearch} & {\tt simulannealbnd} \\
\hline
3 & -22.0429 & -11.1367\\
6 & -16.2076 & -3.1483\\
10 & -9 & -1\\
14& -9 & -1\\
  \hline
\end{tabular}
\caption{{\small Results obtained using {\tt Matlab} global optimization toolbox ({\tt GlobalSearch}
and {\tt simulannealbnd}) on Morse Potential ($d=11$, $N=27$).}} 
\label{tab:matlab} 
\end{table}

\subsection{Nonlinear Systems}
Our technique can also be used to solve nonlinear systems. Indeed, a nonlinear system 
\begin{equation}\label{NLS}
S(x)=0, \mbox{ with } x\in \R^n,
\end{equation}
can be reformulated as an optimization problem, simply by considering the objective function 
given by
\begin{equation}\label{gNLS}
g(x)=\frac{1}{2} S(x)^TS(x)
\end{equation}
In this case, $\nabla g(x)=J_S(x)^TS(x)$, where $J_S$ indicates the Jacobian of $S$.
Therefore, the critical points of the objective function $g$, correspond to 
both the zeros of $S(x)$, and the points for which $S(x)$ is 
in the left null space of the Jacobian. 
\subsubsection{Boggs system}
Given the nonlinear system 
\begin{equation}\label{Boggs}
 \begin{bmatrix} x^2-y+1\\ x-\cos\left(\frac{\pi}{2}y\right)\end{bmatrix}=0, 
 \end{equation}
 we construct the objective function according to \eqref{gNLS}.\\
 The solutions of the problem are $(-1,2)$, $(0,1)$ and $(-\sqrt{2}/2, 3/2)$, 
illustrated in Figure \ref{fig:Boggs-a}.
 \begin{figure}[hbt]
\begin{subfigure}{0.45\textwidth}
\includegraphics[scale=0.4]{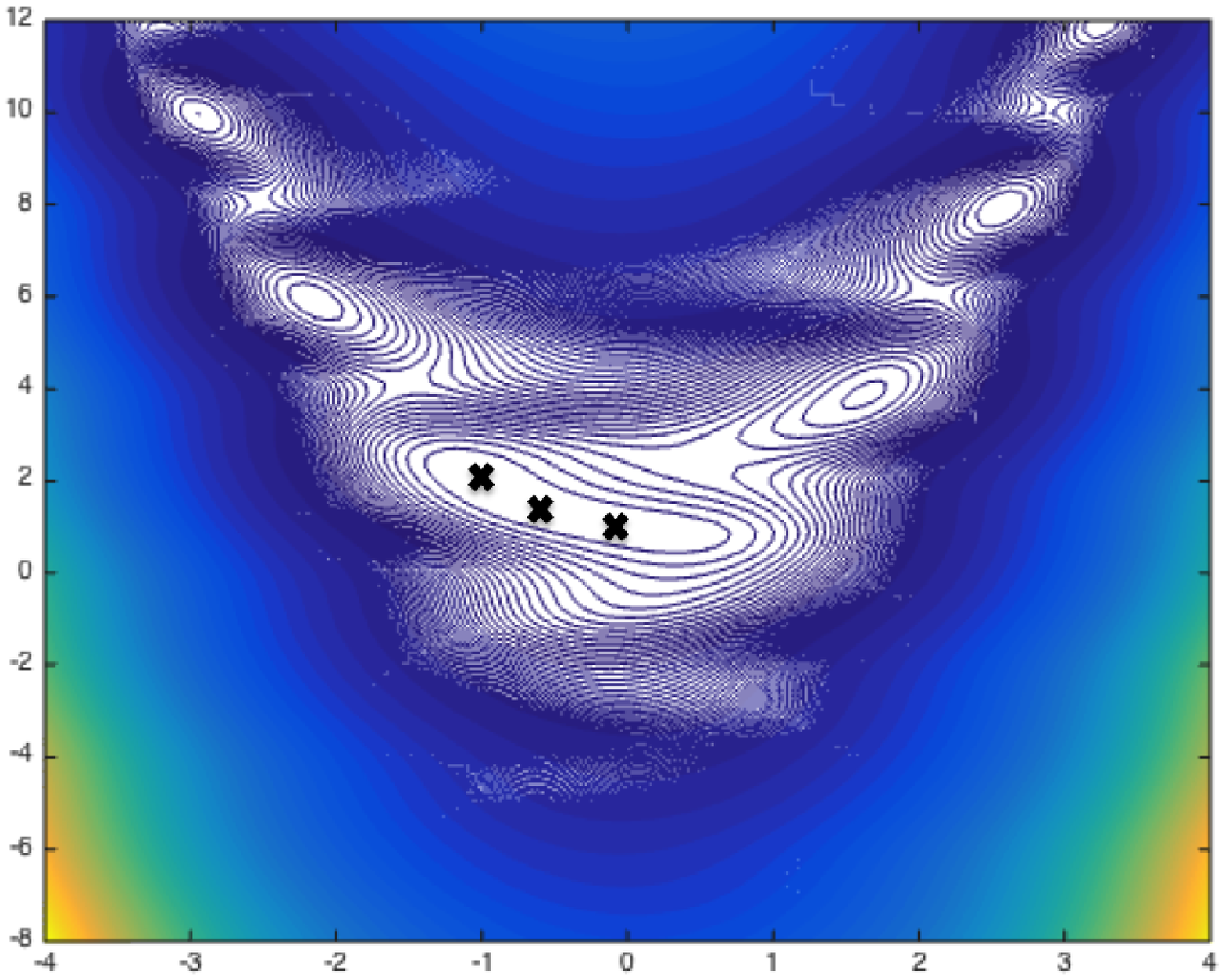}
\vspace{-1cm}
\caption{} \label{fig:Boggs-a}
\end{subfigure}
\begin{subfigure}{0.45\textwidth}
\includegraphics[scale=0.4]{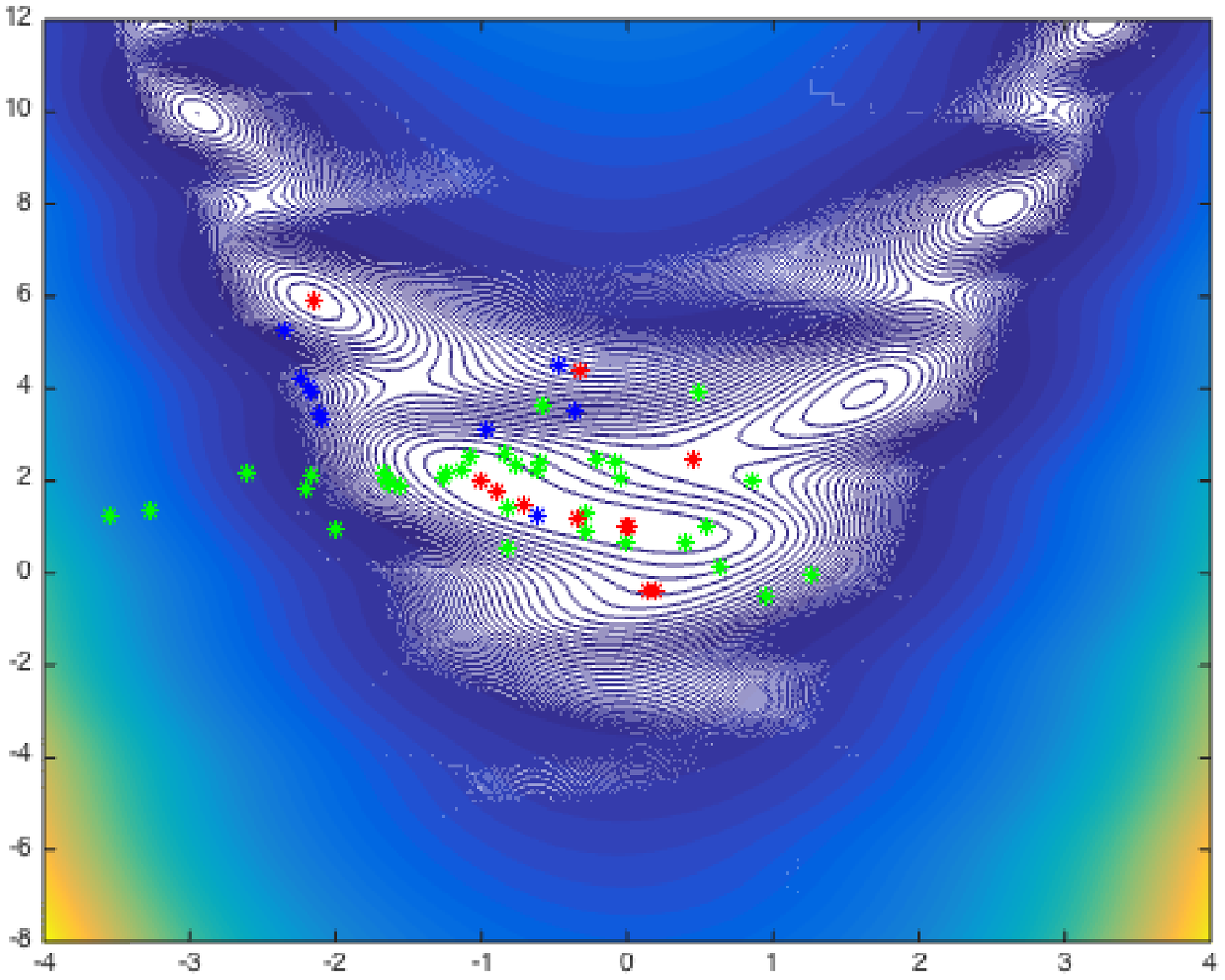}
\vspace{-1cm}
\caption{} \label{fig:Boggs-b}
\end{subfigure}
\caption{\small{Boggs system. (A): contour plot of the potential \eqref{gNLS},
and solutions of the nonlinear system.
(B): critical points of \eqref{gNLS} found by our method; the green dots
correspond to diffusion steps.
}}
\label{fig:Boggs}
\end{figure}

 Numerical results from a single run of the algorithm, asking for at most 20 critical points, 
are shown in Table \ref{tab:Boggs} and illustrated in Figure \ref{fig:Boggs}.
  \begin{table}[htbp] 
\begin{tabular}{||*{10}{c||}|} 
\hline
 $x$ & $y$ & $g(x,y)$ & id \\
 \hline
  0  &  1   & 0& \\
   -1  &  2  &  0& global minima\\
 -0.7071   & 1.5  &  0& \\
 \hline
-2.1530  &  5.9055  &  0.7139 & \\
   0.1301 &     -0.3768 &   1.2161 &  local minima \\
   0.1890 & -0.3663 & 1.1941 &\\
     \hline
         -0.8898 &   1.7671  &  0.0013 & \\
            -0.3319  &  1.1830   & 0.0038&  saddle points \\
             0.4555  &  2.4926   & 1.5111& \\
         -0.3277 & 4.3927 & 6.0502 & \\
 \hline
\end{tabular}
\caption{Critical points - Boggs system } \label{tab:Boggs} 
\end{table}

\section{Conclusions}
In this work, we presented a method apt at exploring the landscape of a smooth (at
least $\mathcal{C}^2$) potential, in order to locate global minima.  The new components
of our method are a double-descent technique (to locate minima) 
and a colored intermittent diffusion (to escape basin of attraction of minima and
other critical points).  The idea of the technique is to use Hessian information
in order to bias the exploration of the landscape.  We illustrated performance of our technique on several
problems from the literature, observing that our method is able, in most cases, to adapt
to different features of the potential.  We believe that our method can
be easily taught in an optimization course, along with other well established
techniques.

\sffamily

\end{document}